\documentclass[12pt]{amsart}
\usepackage{amsmath,amssymb,amscd,xy}

\newtheorem{propo}{Proposition}[section]
\newtheorem{corol}[propo]{Corollary}
\newtheorem{theor}[propo]{Theorem}
\newtheorem{lemma}[propo]{Lemma}

\theoremstyle{definition}
\newtheorem{defin}[propo]{Definition}
\newtheorem{examp}[propo]{Example}

\theoremstyle{remark}
\newtheorem{remar}[propo]{Remark}

\newtheorem{conve}[propo]{Convention}

\numberwithin{equation}{section}

\newcommand{\al }{\alpha }
\newcommand{\cA }{\mathcal{A}}
\newcommand{\cAp }{\mathcal{A}^+}
\newcommand{\cC }{\mathcal{C}}
\newcommand{\Cm }{C}
\newcommand{\cm }{c}
\newcommand{\DG }{\mathbb{D}}
\newcommand{\GLZ }{\mathrm{GL}(2,\ndZ )}
\newcommand{\id }{\mathrm{id}}
\newcommand{\ndN }{\mathbb{N}}
\newcommand{\NN }{\ndN }
\newcommand{\ndR }{\mathbb{R}}
\newcommand{\ndZ }{\mathbb{Z}}
\newcommand{\ZZ }{\ndZ }
\newcommand{\rfl }{\rho }
\newcommand{\Ob }{\mathrm{Ob}}

\newcommand{\rsC }{\mathcal{R}}
\newcommand{\s }{\sigma }
\newcommand{\SLZ }{\mathrm{SL}(2,\ndZ )}
\newcommand{\tr }{\mathrm{tr}}
\newcommand{\trans }{\mathrm{t}}
\newcommand{\Wg }{\mathcal{W}}
\DeclareMathOperator{\Aut}{Aut}
\DeclareMathOperator{\End}{End}
\DeclareMathOperator{\Hom}{Hom}

\xyoption{all}

\title{Weyl groupoids of rank two and continued fractions}

\author{M.~Cuntz}
\address{Michael Cuntz,
Fachbereich Mathematik,
Universit\"at Kaiserslau\-tern,
Postfach 3049,
D-67653 Kaiserslautern, Germany}
\email{cuntz@mathematik.uni-kl.de}

\author{I.~Heckenberger}
\thanks{I.H. is
supported by the German Research Foundation (DFG) via a Heisenberg
fellowship}
\address{Istv\'an Heckenberger, Mathematisches Institut,
Ludwig-Maximili\-ans-Universit\"at M\"unchen,
Theresienstr. 39,
D-80333 M\"unchen, Germany}
\email{i.heckenberger@googlemail.com}

\begin{document}

\begin{abstract}
  A relationship between
  continued fractions
  and
  Weyl groupoids of Cartan schemes of rank two
  is found. This allows to decide easily if a given
  Cartan scheme of rank two admits a finite root system. We obtain
  obstructions and sharp bounds for the entries of the Cartan matrices.
\end{abstract}


\maketitle

\section{Introduction}

Root systems and crystallographic Coxeter groups
appear to be main tools in the study of semisimple Lie algebras
\cite{b-BourLie4-6}. A similar role is expected to be played by Weyl groupoids
and their root systems in the structure theory of pointed Hopf algebras
\cite{b-Montg93}.
The most striking results on pointed Hopf algebras rely on the
Lifting method of Andruskiewitsch and Schneider \cite{a-AndrSchn98}.
Based on it, many new examples of finite-dimensional pointed Hopf algebras
have been detected, and fairly general classification results were achieved
\cite{a-AndrSchn05p}, \cite{p-Heck06b}.
The first step in the Lifting method is the determination of
finite-dimensional Nichols algebras of finite group type.
The upper triangular part of a small quantum group, also called
Frobenius-Lusztig kernel, is a prominent example.
A very natural symmetry object of Nichols algebras of finite group type
is the Weyl groupoid.
This was observed first in \cite{a-Heck06a} for Nichols algebras of diagonal
type, and then in \cite{p-AHS08} in a very general setting. An axiomatic
approach to Weyl groupoids and their root systems, without referring to
Nichols algebras, was initiated in \cite{a-HeckYam08}. The theory includes and
extends the theory of crystallographic Coxeter groups, but contains even such
examples which do not seem to be related to Nichols algebras of diagonal type.
In this paper we use the language and some structural and
classification results achieved in \cite{p-CH08}, see Sect.~\ref{sec:CS} for
the most essential definitions and facts.

For the classification of Nichols algebras of diagonal type it is
crucial to be able to decide,
if a given Cartan scheme (a categorical generalization of the notion
of a generalized Cartan matrix, see Def.~\ref{de:CS}) admits a finite root
system.
Because of the large variety of examples, this seems to be a difficult task.
In our paper, we present a very efficient method for Cartan schemes of
rank two. It relies on a relationship between Cartan schemes of rank two and
continued fractions \cite{b-Perron29}. Instead of giving a complete list of
Cartan schemes of rank two admitting a finite root system (which is then
unique by a result in \cite{p-CH08}),
we present with Thm.~\ref{th:cycleclass} an algorithm. It works with
very elementary operations on sequences of positive integers, and transforms
any Cartan scheme into another one, for which the answer is known.
The algorithm is based on various observations: on the introduction and study
of coverings of Cartan schemes in Sect.~\ref{sec:cover}, on an old theorem
of Stern, Pringsheim, and Tietze,
and a variation of a transformation formula for continued fractions,
see Sect.~\ref{sec:cofrac} and Lemma~\ref{le:cApcontr}, 
on the characterization of simple connected Cartan schemes admitting a finite
root system in terms of certain sequences of positive integers, see
Prop.~\ref{pr:rstoc} and Thm.~\ref{th:ctors},
and on the description of Cartan schemes with object change diagram a cycle
using characteristic sequences, see Def.~\ref{de:centsym}.
As an application, in Sect.~\ref{sec:bounds}
we give obstructions for the entries of the Cartan matrices in a Cartan scheme
admitting a finite root system. We present the power of our method on a small
example at the end of Sect.~\ref{sec:crs2}.

We are confident that a suitable generalization
of our method to Cartan schemes and
Weyl groupoids of higher rank would have a deep impact on the classification
of Nichols algebras, and consider it as a great challenge for the future.

\section{Cartan schemes, root systems, and their Weyl groupoids}
\label{sec:CS}

If not stated otherwise, we follow the notation in \cite{p-CH08}.
Let us start by recalling the main definitions.

Let $I$ be a non-empty finite set and
$\{\al _i\,|\,i\in I\}$ the standard basis of $\ndZ ^I$.
By \cite[\S 1.1]{b-Kac90} a generalized Cartan matrix
$\Cm =(\cm _{ij})_{i,j\in I}$
is a matrix in $\ndZ ^{I\times I}$ such that
\begin{enumerate}
  \item[(M1)] $\cm _{ii}=2$ and $\cm _{jk}\le 0$ for all $i,j,k\in I$ with
    $j\not=k$,
  \item[(M2)] if $i,j\in I$ and $\cm _{ij}=0$, then $\cm _{ji}=0$.
\end{enumerate}

\begin{defin} \label{de:CS}
  Let $A$ be a non-empty set, $\rfl _i : A \to A$ a map for all $i\in I$,
  and $\Cm ^a=(\cm ^a_{jk})_{j,k \in I}$ a generalized Cartan matrix
  in $\ndZ ^{I \times I}$ for all $a\in A$. The quadruple
  \[ \cC = \cC (I,A,(\rfl _i)_{i \in I}, (\Cm ^a)_{a \in A})\]
  is called a \textit{Cartan scheme} if
  \begin{enumerate}
  \item[(C1)] $\rfl _i^2 = \id$ for all $i \in I$,
  \item[(C2)] $\cm ^a_{ij} = \cm ^{\rfl _i(a)}_{ij}$ for all $a\in A$ and
    $i,j\in I$.
  \end{enumerate}
  Two Cartan schemes $\cC =\cC (I,A,(\rfl _i)_{i\in I},(\Cm ^a)_{a\in A})$
  and $\cC '=\cC '(I',A',$
  $(\rfl '_i)_{i\in I'},({\Cm '}^a)_{a\in A'})$
  are termed
  \textit{equivalent}, if there are bijections $\varphi _0:I\to I'$
  and $\varphi _1:A\to A'$ such that
  \begin{align}\label{eq:equivCS}
    \varphi _1(\rfl _i(a))=\rfl '_{\varphi _0(i)}(\varphi _1(a)),
    \qquad
    \cm ^{\varphi _1(a)}_{\varphi _0(i) \varphi _0(j)}=\cm ^a_{i j}
  \end{align}
  for all $i,j\in I$ and $a\in A$.

  Let $\cC = \cC (I,A,(\rfl _i)_{i \in I}, (\Cm ^a)_{a \in A})$ be a
  Cartan scheme. For all $i \in I$ and $a \in A$ define $\s _i^a \in
  \Aut(\ndZ ^I)$ by
  \begin{align}
    \s _i^a (\al _j) = \al _j - \cm _{ij}^a \al _i \qquad
    \text{for all $j \in I$.}
    \label{eq:sia}
  \end{align}
  The \textit{Weyl groupoid of} $\cC $
  is the category $\Wg (\cC )$ such that $\Ob (\Wg (\cC ))=A$ and
  the morphisms are generated by the maps
  $\s _i^a\in \Hom (a,\rfl _i(a))$ with $i\in I$, $a\in A$.
  Formally, for $a,b\in A$ the set $\Hom (a,b)$ consists of the triples
  $(b,f,a)$, where
  \[ f=\s _{i_n}^{\rfl _{i_{n-1}}\cdots \rfl _{i_1}(a)}\cdots
    \s _{i_2}^{\rfl _{i_1}(a)}\s _{i_1}^a \]
  and $b=\rfl _{i_n}\cdots \rfl _{i_2}\rfl _{i_1}(a)$ for some
  $n\in \ndN _0$ and $i_1,\ldots ,i_n\in I$.
  The composition is induced by the group structure of $\Aut (\ndZ ^I)$:
  \[ (a_3,f_2,a_2)\circ (a_2,f_1,a_1) = (a_3,f_2f_1, a_1)\]
  for all $(a_3,f_2,a_2),(a_2,f_1,a_1)\in \Hom (\Wg (\cC ))$.
  By abuse of notation we will write
  $f\in \Hom (a,b)$ instead of $(b,f,a)\in \Hom (a,b)$.
 
  The cardinality of $I$ is termed the \textit{rank of} $\Wg (\cC )$.
  A Cartan scheme is called \textit{connected} if its Weyl groupoid
  is connected, that is, if for all $a,b\in A$ there exists $w\in \Hom (a,b)$.
\end{defin}

In many cases it will be natural to assume that a Cartan scheme satisfies the
following additional property.
\begin{itemize}
  \item [(C3)] If $a,b\in A$ and $(b,\id ,a)\in \Hom (a,b)$, then $a=b$.
\end{itemize}

\begin{defin} \label{de:RSC}
  Let $\cC =\cC (I,A,(\rfl _i)_{i\in I},(\Cm ^a)_{a\in A})$ be a Cartan
  scheme. For all $a\in A$ let $R^a\subset \ndZ ^I$, and define
  $m_{i,j}^a= |R^a \cap (\ndN_0 \al _i + \ndN_0 \al _j)|$ for all $i,j\in
  I$ and $a\in A$. We say that
  \[ \rsC = \rsC (\cC , (R^a)_{a\in A}) \]
  is a \textit{root system of type} $\cC $, if it satisfies the following
  axioms.
  \begin{enumerate}
    \item[(R1)]
      $R^a=R^a_+\cup - R^a_+$, where $R^a_+=R^a\cap \ndN_0^I$, for all
      $a\in A$.
    \item[(R2)]
      $R^a\cap \ndZ\al _i=\{\al _i,-\al _i\}$ for all $i\in I$, $a\in A$.
    \item[(R3)]
      $\s _i^a(R^a) = R^{\rfl _i(a)}$ for all $i\in I$, $a\in A$.
    \item[(R4)]
      If $i,j\in I$ and $a\in A$ such that $i\not=j$ and $m_{i,j}^a$ is
      finite, then
      $(\rfl _i\rfl _j)^{m_{i,j}^a}(a)=a$.
  \end{enumerate}
  If $\rsC $ is a root system of type $\cC $, then we say that
  $\Wg (\rsC )=\Wg (\cC )$ is the \textit{Weyl groupoid of} $\rsC $.
  Further, $\rsC $ is called \textit{connected}, if $\cC $ is a connected
  Cartan scheme.
  If $\rsC =\rsC (\cC ,(R^a)_{a\in A})$ is a root system of type $\cC $
  and $\rsC '=\rsC '(\cC ',({R'}^a)_{a\in A'})$ is a root system of
  type $\cC '$, then we say that $\rsC $ and $\rsC '$ are \textit{equivalent},
  if $\cC $ and $\cC '$ are equivalent Cartan schemes given by maps $\varphi
  _0:I\to I'$, $\varphi _1:A\to A'$ as in Def.~\ref{de:CS}, and if
  the map $\varphi _0^*:\ndZ^I\to \ndZ^{I'}$ given by
  $\varphi _0^*(\al _i)=\al _{\varphi _0(i)}$ satisfies
  $\varphi _0^*(R^a)={R'}^{\varphi _1(a)}$ for all $a\in A$.
\end{defin}

There exist many interesting examples of root systems of type $\cC $ related
to semisimple Lie algebras, Lie superalgebras and Nichols algebras of diagonal
type, respectively. For further details and results we refer to
\cite{a-HeckYam08} and \cite{p-CH08}.

\begin{conve}\label{con:uind}
  In connection with Cartan schemes, upper indices usually refer to elements
  of $A$. Often, these indices will be omitted if they are uniquely determined
  by the context.
\end{conve}

\begin{remar}
  If $\cC $ is a Cartan scheme and there exists a root system of type $\cC $,
  then $\cC $ satisfies (C3) by \cite[Lemma\,8(iii)]{a-HeckYam08}.
\end{remar}

In \cite[Def.\,4.3]{p-CH08} the concept of an \textit{irreducible}
root system of type
$\cC $ was defined. By \cite[Prop.\,4.6]{p-CH08}, if $\cC $ is a connected
Cartan scheme and $\rsC $ is a finite root system of type $\cC $, then $\rsC $
is irreducible if and only if the generalized Cartan matrix $C^a$ is
indecomposable for one (equivalently, for all) $a\in A$.

A fundamental result about Weyl groupoids is the following theorem.

\begin{theor}\cite[Thm.\,1]{a-HeckYam08}\label{th:Coxgr}
  Let $\cC =\cC (I,A,(\rfl _i)_{i\in I},(\Cm ^a)_{a\in A})$
  be a Cartan scheme and $\rsC =\rsC (\cC ,(R^a)_{a\in A})$ a root system
  of type $\cC $.
  Let $\Wg $ be the abstract
  groupoid with $\Ob (\Wg )=A$ such that $\Hom (\Wg )$ is
  generated by abstract morphisms $s_i^a\in \Hom (a,\rfl _i(a))$,
  where $i\in I$ and $a\in A$, satisfying the relations
  \begin{align*}
    s_i s_i 1_a=1_a,\quad (s_j s_k)^{m_{j,k}^a}1_a=1_a,
    \qquad a\in A,\,i,j,k\in I,\, j\not=k,
  \end{align*}
  see Conv.~\ref{con:uind}.
  Here $1_a$ is the identity of the object $a$,
  and $(s_j s_k)^\infty 1_a$ is understood to be
  $1_a$. The functor $\Wg \to \Wg (\rsC )$, which is
  the identity on the objects, and on the set of
  morphisms is given by
  $s _i^a\mapsto \s_i^a$ for all $i\in I$, $a\in A$,
  is an isomorphism of groupoids.
\end{theor}

\begin{defin}
  Let $\cC =\cC (I,A,(\rfl _i)_{i\in I},(\Cm ^a)_{a\in A})$ be a Cartan
  scheme.
  Let $\Gamma $ be a nondirected graph,
  such that the vertices of $\Gamma $ correspond to the elements of $A$.
  Assume that for all $i\in I$ and $a\in A$ with $\rfl _i(a)\not=a$
  there is precisely one edge between the vertices $a$ and $\rfl _i(a)$
  with label $i$,
  and all edges of $\Gamma $ are given in this way.
  The graph $\Gamma $ is called the \textit{object change diagram} of $\cC $.
  If $\rsC =\rsC (\cC ,(R^a)_{a\in A})$ is a root system of type $\cC $, then
  we also say that $\Gamma $ is the object change diagram of $\rsC $.
\end{defin}

\section{Coverings of Cartan schemes, Weyl groupoids, and root systems}
\label{sec:cover}

Two Cartan schemes can be related to each other in different ways.
In this section we analyze coverings of Cartan schemes. The definition
is motivated by the corresponding notion in topology.

\begin{defin} \label{de:cover}
  Let $\cC =\cC (I,A,(\rfl _i)_{i\in I},(C^a)_{a\in A})$ and
  $\cC' =\cC' (I,A'$, $(\rfl '_i)_{i\in I},(C'^a)_{a\in A'})$ be
  connected Cartan schemes.
  Let $\pi :A'\to A$ be a map such that $C^{\pi (a)}=C'^a$
  for all $a\in A'$ and the diagrams
  \begin{align}
    \begin{CD}
      A' & @>\rfl '_i>> & A' \\
      @V\pi VV & & @VV\pi V\\
      A & @>>\rfl _i > & A
    \end{CD}
    \label{eq:pirfl}
  \end{align}
  commute for all $i\in I$. We say that $\pi :\cC '\to \cC $
  is a \textit{covering}, and that $\cC '$ is a \textit{covering of} $\cC $.
\end{defin}

The composition of two coverings is
again a covering.
Any covering $\pi :\cC '\to \cC $ of Cartan schemes $\cC ',\cC $ is
surjective by \eqref{eq:pirfl},
since $A'$ is non-empty and $\cC $ is connected.

\begin{remar}
  Many of the following results can be formulated without assuming that $\cC $
  and/or $\cC '$ in Def.~\ref{de:cover} are connected Cartan schemes.
  In that case one should assume
  that $\pi $ is a surjective map. However, in 
  the applications we are interested in, all Cartan schemes are connected, and
  hence we prefer the above definition in order to simplify the terminology.
\end{remar}

Any covering $\pi :\cC '\to \cC $ of Cartan schemes $\cC ',\cC $
induces a covariant
functor $F_\pi :\Wg (\cC ')\to \Wg (\cC )$
by letting
\[ F_\pi (a')=\pi (a'), \quad F_\pi (\s _i^{a'})=\s _i^{\pi (a')}\qquad
\text{for all $i\in I$, $a'\in A'$.} \]
In this case the Weyl groupoid $\Wg (\cC ')$ is termed a \textit{covering of}
$\Wg (\cC)$, and the functor $F_\pi $ a covering of Weyl groupoids.

First we need a technical result.

\begin{lemma}
  Let $\pi :\cC '\to \cC $ be a covering, and assume that $\cC '$ satisfies
  Axiom~(C3). Then the following hold:

  (1) $\cC $ satisfies (C3).

  (2) Let $a\in A$ and $a',a''\in A'$ such that $\pi (a')=\pi (a'')=a$.
  If $w'\in \Hom (a',a'')$ such that $F_\pi (w')\in F_\pi (\End (a'))$,
  then $a'=a''$.
  \label{le:C3}
\end{lemma}

\begin{proof}
  (1) Let $a\in A$.
  If $k\in \ndN _0$ and $i_1,\ldots ,i_k\in I$, then Def.~\ref{de:cover}
  gives that
  $\s _{i_1}\cdots \s _{i_{k-1}}\s _{i_k}^a=
  \s _{i_1}\cdots \s _{i_{k-1}}\s _{i_k}^{a'}$ in $\Aut (\ndZ ^I)$
  for all $a'\in A'$ with $\pi (a')=a$. Assume now that
  $\s _{i_1}\cdots \s _{i_{k-1}}\s _{i_k}^a=\id $. Then
  $\rfl '_{i_1}\cdots \rfl '_{i_k}(a')=a'$ for all $a'\in A'$ with $\pi
  (a')=a$, since $\cC '$ satisfies (C3).
  Hence $\rfl _{i_1}\cdots \rfl _{i_k}(a)=a$ by Eq.~\eqref{eq:pirfl}.
  This yields the claim.

  (2)
  Let $w''\in \End (a')$ with $F_\pi (w'')=F_\pi (w')$. Then $F_\pi
  (w'w''^{-1})=\id _a$, and hence $w'w''^{-1}=\id $ in $\Aut (\ndZ ^I)$.
  Since $\cC '$ satisfies (C3), it follows that $w'w''^{-1}=\id _{a'}$, and
  hence $a'=a''$.
\end{proof}

Let $\cC =\cC (I,A,(\rfl _i)_{i\in I},(C^a)_{a\in A})$
be a connected Cartan-scheme, $\Wg (\cC )$ its Weyl groupoid, and $a\in A$.
Coverings of $\cC $ can be parametrized by
subgroups of $\End (a)\subset \Hom (\Wg (\cC ))$ (up to conjugation).

\begin{propo}\label{pr:cover}
  (1) Let $\cC '$ be a connected Cartan scheme
  and assume that $\pi :\cC '\to \cC $ is a covering.
  Let $a'\in A'$ with $\pi (a')=a$.
  \begin{itemize}
    \item[(1.A)]
      The group homomorphism $F_\pi :\End (a')\to \End (a)$ is injective.
    \item[(1.B)]
      For each $b'\in A'$ with $\pi (b')=a$ the subgroup $F_\pi (\End
      (b'))$ of $\End (a)$ is conjugate to $F_\pi (\End (a'))$.
    \item[(1.C)]
      If $U'$ is a subgroup of $\End (a)$ conjugate to $F_\pi (\End (a'))$,
      then there exists
      $b'\in A'$ with $\pi (b')=a$ and $F_\pi (\End (b'))=U'$.
  \end{itemize}

  (2) Suppose that $U\subset \End (a)$ is a subgroup.
  Then there exists a covering $\pi :\cC '\to \cC $
  and $b'\in A'$ such that
  \begin{align}
    F_\pi (\End (b'))=&U,
    \label{eq:piprop1}\\
    |\pi ^{-1}(b)|=&[\End (a):U] \quad \text{for all $b\in A$.}
    \label{eq:piprop2}
  \end{align}
  If $\cC $ satisfies Axiom~(C3),
  then up to equivalence there is a unique covering $\cC '$
  satisfying Eq.~\eqref{eq:piprop1} 
  and Axiom~(C3). For this covering Eq.~\eqref{eq:piprop2} holds.
\end{propo}

\begin{proof}
%
  (1.A) Each element $w'\in \End (a')$ is a product of $\s _i^{b'}$ for some
  $i\in I$ and $b'\in A'$. Moreover, $w'$ can be naturally regarded as an
  element in $\Aut (\ndZ ^I)$. The same is true for $w\in \End (a)$.
  Since $C'^{b'}=C^{\pi (b')}$ for all $b'\in
  A'$, $F_\pi (w')$ identifies with the same element of $\Aut (\ndZ ^I)$ as
  $w'$. This proves (1.A).

  (1.B) Let $b'\in A'$. Since $\cC '$ is connected,
  there exists $w'\in \Hom (a',b')$.
  Then $\End (b')=w'\End (a')w'^{-1}$. Since $F_\pi $ is a functor,
  $F_\pi (\End (b'))=F_\pi (w')F_\pi (\End (a'))F_\pi (w')^{-1}$.

  (1.C) Assume that $w\in \End (a)$ such that $U'=wF_\pi (\End (a'))w^{-1}$.
  Then
  $w=\s _{i_1}\cdots \s _{i_{k-1}}\s _{i_k}^a$ for some $k\in \ndN _0$ and
  $i_1,\ldots ,i_k\in I$. Let
  $w'=\s _{i_1}\cdots \s _{i_{k-1}}\s _{i_k}^{a'}$ and
  $b'=\rfl '_{i_1}\cdots \rfl '_{i_k}(a')$. Then $\End (b')=w'\End (a')w'^{-1}$,
  and hence $F_\pi (\End (b'))=w F_\pi (\End (a'))w^{-1}=U'$.

  (2) We construct $\cC '$ explicitly. Let
  \[ A'=\Hom (\Wg (\cC ))/U
  = \{ gU\subset \Hom (a,b)\,|\,b\in A,g\in \Hom (a,b)\}
  \]
  be the set of left cosets. For all $i\in I$ and $gU\in A'$ with
  $g\in \Hom (a,b)$, where $b\in A$,
  define $C'^{gU}=C^b$ and $\rfl '_i(gU)=\s _i^bgU$.
  Then $\rfl '_i:A'\to A'$ satisfies (C1) since
  $\s _i^{\rfl _i(b)}\s _i^b=\id $ and $\rfl _i^2=\id $, and $\cC '$ fulfills
  (C2), since $\cC $ does. Since $\cC $ is connected, $\cC
  '=\cC '(I,A',(\rfl '_i)_{i\in I},(C'^{a'})_{a'\in A'})$ is a connected
  Cartan scheme.
  Define $\pi :A'\to A$
  by $\pi (gU)=b$ for all $b\in A$, $g\in \Hom (a,b)$. Then
  $F_\pi (\End (1_aU))=U$ and
  $|\pi ^{-1}(a)|=[\End (a):U]$. Since $\cC '$ is connected,
  $|\pi ^{-1}(b)|=|\pi ^{-1}(a)|$ for all $b\in A$.

  Assume that $\cC $ satisfies (C3). We show that $\cC '$ satisfies (C3).
  For $l\in \{1,2\}$ let $a_l\in A$ and $g_l\in \Hom (a,a_l)$
  such that
  $(g_1U,\id ,g_2U)\in \Hom (\Wg (\cC '))$. Then there exist
  $k\in \ndN _0$ and $i_1,\ldots ,i_k\in I$ such that
  $\s _{i_1}\cdots \s _{i_{k-1}}\s _{i_k}^{a_2}g_2U=g_1U$
  and that
  $\s _{i_1}\cdots \s _{i_{k-1}}\s _{i_k}^{a_2}=\id $ in $\Aut (\ndZ ^I)$.
  Since $\cC $ fulfills (C3), we obtain that $a_1=a_2$, and hence $g_2U=g_1U$.
  Therefore $\cC '$ satisfies (C3).

  Finally, let $\pi :\cC '\to \cC $ and $\pi '':\cC ''\to \cC $ be
  coverings of $\cC $ satisfying (C3), and assume that there exist
  $b'\in A'$, $b''\in A''$ such that $\pi (b')=\pi ''(b'')=a$ and
  $F_\pi (\End (b'))=F_{\pi ''}(\End (b''))=U$.
  We have to show that $\cC '$ and $\cC ''$ are equivalent Cartan schemes.
  Define $\phi :A'\to A''$ by
  \[ \phi (\rfl '_{i_1}\cdots \rfl '_{i_k}(b'))=
  \rfl ''_{i_1}\cdots \rfl ''_{i_k}(b'')\qquad
  \text{for all $k\in \ndN _0$, $i_1,\ldots ,i_k\in I$.}
  \]
  Then $\phi $ is well-defined. Indeed,
  assume that $\rfl '_{i_1}\cdots \rfl '_{i_k}(b')=b'$.
  Then
  $\s _{i_1}\cdots \s _{i_k}^{b'}\in \End (b')$, and hence application of
  $\pi $ resp. $F_\pi $ gives that
  $\rfl _{i_1}\cdots \rfl _{i_k}(a)=a$,
  $\s _{i_1}\cdots \s _{i_k}^a\in U$.
  Thus
  $F_{\pi ''}(\s _{i_1}\cdots \s _{i_k}^{b''})\in U$, and hence
  Lemma~\ref{le:C3}(2) gives that
  $\rfl ''_{i_1}\cdots \rfl ''_{i_k}(b'')=b''$.
  The compatibility of $\phi $ with $\rfl ',\rfl '',C'^{b'},C''^{b''}$ is
  fulfilled by Def.~\ref{de:cover} and by definition of $\phi $.
  Further, $\phi :A'\to A''$ is a bijection, the construction of $\phi ^{-1}$
  being analogous.
  Hence $\phi $
  gives rise to an equivalence of the Cartan schemes $\cC '$ and $\cC ''$.
\end{proof}

\begin{defin}
  We say that a Cartan scheme $\cC $ is \textit{simply connected},
  if $\End (a)$ is the trivial group for all $a\in A$.
\end{defin}

\begin{corol}
  Let $\cC $ be a connected Cartan scheme satisfying (C3).
  Then up to equivalence there exists a unique
  covering $\cC '$ of $\cC $ which is simply connected and satisfies (C3).
\end{corol}

As usual, this simply connected covering of $\cC $ is called the
\textit{universal covering}.

\begin{proof}
  The claim follows from 
  Prop.~\ref{pr:cover}(2) by setting $U=\{1\}$.
\end{proof}

\begin{propo}\label{pr:covrs}
  Let $\cC $, $\cC '$ be connected Cartan schemes and
  $\pi :\cC '\to \cC $ a covering.
  
  (1) If there exists a root system $\rsC '$
  of type $\cC '$, then the equations
  \begin{align}
    R^a=\mathop{\bigcap}_{a'\in A'\,|\,\pi (a')=a}R'^{a'}
    \qquad \text{for all $a\in A$}
    \label{eq:RR'1}
  \end{align}
  define a root system $\rsC $ of type $\cC $.

  (2) If there exists a root system $\rsC $
  of type $\cC $, and $\cC '$ satisfies (C3),
  then the equations
  \begin{align}
    R'^{a'}=R^{\pi (a')} \qquad \text{for all $a'\in A'$}
    \label{eq:RR'2}
  \end{align}
  define a root system $\rsC '$ of type $\cC '$.
\end{propo}

\begin{proof}
  (1) By Def.~\ref{de:cover} and Axioms~(R1)--(R4) for $\rsC '$,
  the axioms (R1)--(R4) are fulfilled for $\rsC $.

  (2) Since Axioms~(R1)--(R3) hold for $\rsC $, they also hold for $\rsC '$.
  Suppose that $i,j\in I$ and $a'\in A'$
  such that $i\not=j$ and that $m^{a'}_{i,j}=m^a_{i,j}$ is finite, where
  $a=\pi (a')$.
  Then $(\s _i\s _j)^{m^a_{i,j}}1_a=\id _a$ by Thm.~\ref{th:Coxgr}. Hence 
  $(\s _i\s _j)^{m^a_{i,j}}1_{a'}=\id $, and (C3) for $\cC '$
  implies that $(\rfl '_i\rfl '_j)^{m^{a'}_{i,j}}(a')=a'$. Thus (R4) holds for
  $\rsC '$ and hence $\rsC '$ is a root system of type $\cC '$.
\end{proof}

\section{Continued fractions}
\label{sec:cofrac}

Continued fractions are related to Weyl groupoids of Cartan schemes of rank two.
We recall some basic facts about continued fractions and formulate the facts
we will use in our study.

A {\it continued fraction} is a sequence of indeterminates
$a_1,a_2,a_3,\ldots$, $b_0,b_1,b_2,\ldots$
written in the form
\[ b_0+\frac{a_1|}{|b_1}+\frac{a_2|}{|b_2}+\ldots =
b_0+\frac{a_1}{b_1+\frac{a_2}{b_2+ _{\ddots} }} \]
(see \cite{b-Perron29} for an introduction). Specializing the right expression
to integers, the {\it convergents} are the numbers
\[ \frac{A_n}{B_n} = b_0+\frac{a_1|}{|b_1}+\frac{a_2|}{|b_2}
+\ldots+\frac{a_n|}{|b_n}, \]
also given by the recursion
\begin{gather*}
  A_{-1}=1, \quad A_0=b_0, \qquad B_{-1}=0, \quad B_0=1, \\
  A_\nu=b_\nu A_{\nu-1}+a_\nu A_{\nu-2},\quad
  B_\nu=b_\nu B_{\nu-1}+a_\nu B_{\nu-2}
\end{gather*}
for all $\nu \in \ndN $, or
\begin{eqnarray}
\label{etagen0}
\begin{pmatrix} B_0 & A_0 \\ B_{-1} & A_{-1} \end{pmatrix} &=&
\begin{pmatrix} 1 & b_0 \\ 0 & 1 \end{pmatrix},\\
\label{eq:etagen}
\begin{pmatrix} b_\nu & a_\nu \\ 1 & 0 \end{pmatrix}
\begin{pmatrix} B_{\nu-1} & A_{\nu-1} \\ B_{\nu-2} & A_{\nu-2} \end{pmatrix} &=&
\begin{pmatrix} B_\nu & A_\nu \\ B_{\nu-1} & A_{\nu-1} \end{pmatrix}.
\end{eqnarray}
One says that 
$b_0+\frac{a_1|}{|b_1}+\frac{a_2|}{|b_2}+\ldots $ is \textit{convergent}, if
the sequence $(A_\nu /B_\nu )_{\nu \ge \nu _0}$ is well-defined and convergent
(with respect to the standard topology of $\ndR $) for some $\nu _0\in \ndN $.

The case where all $a_\nu$ are $1$ is the most important one and
well understood.
However, we will be interested in a different case:
From now on, let $a_\nu=-1$, $b_\nu\in\NN$ for all $\nu$ and assume
that the sequence $b_1,b_2,\ldots$ is periodic.
For any $i\in \ZZ$, let
\begin{align}
  \eta (i)=
  \begin{pmatrix}
    i & -1 \\ 1 & 0
  \end{pmatrix} \in \SLZ .
  \label{eq:etai}
\end{align}
We will often need the following equations, which hold for all $i,j,k\in \ndZ
$.
\begin{align}
  \eta (i)^{-1}=&
  \begin{pmatrix}
    0 & 1 \\ -1 & i
  \end{pmatrix},
  \label{eq:etainv}\\
  \eta (i)\eta (j)=&
  \begin{pmatrix}
    ij-1 & -i \\ j & -1
  \end{pmatrix},
  \label{eq:eta2}\\
  \eta (i)\eta (j)\eta (k)=&
  \begin{pmatrix}
    (ij-1)k-i & -(ij-1) \\ jk-1 & -j
  \end{pmatrix},
  \label{eq:eta3}\\
  \tau \eta (i)\tau =& \eta (i)^{-1},\quad
  \tau \eta (i)^{-1}\tau =\eta (i),
  \label{eq:etatau}
\end{align}
where
\begin{align}
  \tau =
  \begin{pmatrix}
    0 & 1 \\ 1 & 0
  \end{pmatrix}.
  \label{eq:tau}
\end{align}
By Eq.~\eqref{eq:etagen},
\[ \begin{pmatrix} B_n \\ B_{n-1} \end{pmatrix} =
\eta(b_n)\cdots\eta(b_1)
\begin{pmatrix} B_0 \\ B_{-1} \end{pmatrix}. \]
The product $\eta(b_n)\cdots\eta(b_1)$ will appear in the study of
Weyl groupoids of rank two. In particular, we will need to know for
which sequences $b_n,\ldots,b_1$ this product has finite order.
If it has finite order, then, since $B_{-1}=0$,
there exists $\nu \in \ndN $ such that $B_\nu =0$.

The following fact is well-known. Variations of it were considered for example
by Stern \cite[\S 51,\,Satz\,15]{b-Perron29}, Pringsheim
\cite[\S 53,\,Satz\,24]{b-Perron29} and Tietze
\cite[\S 35,\,Satz\,1]{b-Perron29}.

\begin{theor}\label{th:Tietze}
If $a_\nu =-1$ and $b_\nu \ge 2$
for all $\nu \in \ndN $, then the continued fraction
$\frac{a_1|}{|b_1}+\frac{a_2|}{|b_2}+\ldots$
is convergent.
\end{theor}

Thus we get:

\begin{corol}\label{co:Pringsheim}
Let $n\in\NN$ and $b_1,\ldots,b_n\in\ZZ$. If $b_i\ge 2$
for all $i\in\{1,\ldots,n\}$, then $\eta(b_1)\cdots\eta(b_n)$
does not have finite order.
\end{corol}
\begin{proof}
Assume $b_i\ge 2$ for all $i\in\{1,\ldots,n\}$.
If $\eta(b_1)\cdots\eta(b_n)$ had finite order, then
the periodic continued fraction
\[ \frac{-1|}{|b_n}+\frac{-1|}{|b_{n-1}}+\cdots +\frac{-1|}{|b_1}
   +\frac{-1|}{|b_n}+\frac{-1|}{|b_{n-1}}+\cdots +\frac{-1|}{|b_1}
   +\frac{-1|}{|b_n}+\cdots 
\]
would have infinitely many convergents with denominator $0$.
This is a contradiction to Thm.~\ref{th:Tietze}.
\end{proof}

One can also prove Cor.~\ref{co:Pringsheim} without
Thm.~\ref{th:Tietze}, e.\,g.~by \cite[Lemma\,9]{a-Heck08a}.

\section{Distinguished finite sequences of integers}
\label{sec:seq}

We now study a special class of finite sequences of positive integers.
They correspond to a class of continued fractions which are not convergent.
Later we will use these sequences to classify finite root systems of type
$\cC $ and rank two. Recall the definition of the map $\eta :\ndZ \to \SLZ $
from Eq.~\eqref{eq:etai}.

\begin{defin}
  Let $\cA $ denote the set of finite sequences $(c_1,\ldots ,c_n)$ of
  integers such that $n\ge 1$ and $\eta (c_1)\cdots \eta (c_n)=-\id $.
  Let $\cAp $ be the subset of $\cA $ formed by those $(c_1,\ldots ,c_n)\in
  \cA $, for which $c_i\ge 1 $ for all $i\in \{1,\ldots ,n\}$ and
  the entries in the first column of $\eta (c_1)\cdots \eta (c_i)$ are
  nonnegative for all $i<n$.
\end{defin}

The following lemma will be crucial for our analysis of $\cAp $.
It is related to a well-known transformation formula for continued fractions,
see \cite[\S 37,\,Eqs.\,(1),(2)]{b-Perron29}.

\begin{lemma}
  \label{le:cApcontr}
  Let $n\ge 3$ and $c=(c_1,1,c_3,c_4,\ldots ,c_n)$
  such that $c_i\in \ndZ $ for all
  $i\in \{1,\dots ,n\}$. Let $c'=(c_1-1,c_3-1,c_4,\dots ,c_n)$.
  
  (1) $c'\in \cA $ if and only if $c\in \cA $.

  (2) $c'\in \cAp $ if and only if $c\in \cAp $, $c_1,c_3\ge 2$.

  (3) If $c\in \cAp $, then either $n=3$, $c_1=c_3=1$ or $n>3$,
  $c_1,c_3\ge 2$.
\end{lemma}

\begin{proof}
  If $i,k\in \ndZ $, then
  \begin{align*}
    \eta (i)\eta (1)\eta (k)=
    \begin{pmatrix}
      ik-i-k & 1-i \\ k-1 & -1
    \end{pmatrix}
    =\eta (i-1)\eta (k-1)
  \end{align*}
  by Eqs.~\eqref{eq:eta2}, \eqref{eq:eta3}. This gives (1).
  By Eq.~\eqref{eq:eta2}, the first column of $\eta (c_1)\eta (1)$
  contains only nonnegative
  integers if and only if $c_1\ge 1$. Thus (2) holds.
  Let $c\in \cAp $ such that $c_1=1$ or $c_3=1$.
  Then Eq.~\eqref{eq:eta3} gives that the upper left entry of $\eta (c_1)\eta
  (1)\eta (c_3)$ is $-1$, and hence $n=3$. Then $c\in \cA $ implies that
  $c_1=c_3=1$. Hence (3) is proven.
\end{proof}

\begin{propo}
  Let $n\in \ndN $ and $(c_1,\ldots ,c_n)\in \cAp $.
  
  (1) Let $i,j\in \{1,\ldots ,n\}$ with $i\le j$ and $(i,j)\not=(1,n)$.
  Then $\eta (c_i)\eta (c_{i+1})\cdots \eta (c_j)\in \SLZ $
  such that the first column contains only
  nonnegative and the second only nonpositive integers.

  (2) Let $i\in \{1,\ldots ,n\}$.
  Then $(c_i,c_{i+1},\ldots ,c_n,c_1,\ldots ,c_{i-1})\in \cAp $.

  (3) $(c_n,c_{n-1},\ldots ,c_2,c_1)\in \cAp $.
  
  (4) If $n\le 3$ then $(c_1,\ldots ,c_n)=(1,1,1)$.
  \label{pr:Aprop}
\end{propo}

\begin{proof}
  (1) We proceed by induction on the lexicographically ordered pairs $(i,j)$.

  If $i=j$ then we are done, since the matrix $\eta (c_i)$ satisfies the
  claim.

  Let $i,j\in \{1,\ldots ,n\}$ with $i<j$ and $(i,j)\not=(1,n)$.
  Assume
  that the claim holds for all pairs $(i',j')\in \{1,\ldots ,n\}$ such that
  $i'\le j'$ and either $i'<i$ or $i'=i$, $j'<j$. Let
  \[ \eta (c_i)\cdots \eta (c_j)=
  \begin{pmatrix}
    a & -b \\ c & -d
  \end{pmatrix}
  \]
  with $a,b,c,d\in \ndZ $. Clearly, $-ad+bc=1$ since $\eta (k)\in
  \SLZ $ for all $k\in \ndZ $.
  Moreover, Eq.~\eqref{eq:etainv} gives that
  \[ \eta (c_i)\cdots \eta (c_{j-1})=
  \begin{pmatrix}
    a & -b \\ c & -d
  \end{pmatrix}
  \begin{pmatrix}
    0 & 1 \\ -1 & c_j
  \end{pmatrix}
  =
  \begin{pmatrix}
    b & -(bc_j-a) \\ d & -(dc_j-c)
  \end{pmatrix}.
  \]
  Hence $b,d\ge 0$ by induction hypothesis.

  If $i=1$, then $a,c\ge 0$ by definition of $\cAp $ and the assumption
  $(i,j)\not=(1,n)$, and hence we are done.
  Otherwise
  \[ \eta (c_{i-1})\cdots \eta (c_j)=
  \begin{pmatrix}
    c_{i-1} & -1 \\ 1 & 0
  \end{pmatrix}
  \begin{pmatrix}
    a & -b \\ c & -d
  \end{pmatrix}
  =
  \begin{pmatrix}
    c_{i-1}a-c & d-c_{i-1}b \\ a & -b
  \end{pmatrix},
  \]
  and hence $a>0$ by induction hypothesis.
  Since $a,b,d\ge 0$, we get $bc=1+ad\ge 1$, and hence $c>0$, which proves the
  claim.

  (2) It suffices to prove the claim for $i=2$.
  If $\eta (c_1)\cdots \eta (c_n)=-\id $, then clearly
  $\eta (c_2)\cdots \eta (c_n)\eta (c_1)=-\id $.
  Let $j\in \{2,\ldots ,n\}$. Then
  the entries in the first column of $\eta (c_2)\cdots \eta (c_j)$
  are nonnegative by Part~(1) of the proposition.
  This gives (2).

  (3) Recall the definition of $\tau $ in Eq.~\eqref{eq:tau}.
  Then Eq.~\eqref{eq:etatau} gives that
  \[
    \eta (c_n)\eta (c_{n-1})\cdots \eta (c_1)=
    \tau \eta (c_n)^{-1}\eta (c_{n-1})^{-1}\cdots \eta (c_1)^{-1}\tau
    =-\id
  \]
  since $\eta (c_1)\cdots \eta (c_n)=-\id $. Therefore $(c_n,c_{n-1},\dots
  ,c_1)\in \cA $.

  Let $2\le i\le n$ and assume that
  \[ \eta (c_i)\eta (c_{i+1})\cdots \eta (c_n)=
    \begin{pmatrix}
      a & -b \\ c & -d
    \end{pmatrix}
  \]
  for some $a,b,c,d\in \ndZ $. Then $a,b,c,d\ge 0$ and $bc-ad=1$
  by Part~(1) of the proposition. We obtain that
  \begin{align*}
    \eta (c_n)\cdots \eta (c_i)=&
    \tau \eta (c_n)^{-1}\cdots \eta (c_i)^{-1}\tau \\
    =&
    \begin{pmatrix}
      0 & 1 \\ 1 & 0
    \end{pmatrix}
    \begin{pmatrix}
      -d & b \\ -c & a
    \end{pmatrix}
    \begin{pmatrix}
      0 & 1 \\ 1 & 0
    \end{pmatrix}
    =
    \begin{pmatrix}
      a & -c \\ b & -d
    \end{pmatrix}.
  \end{align*}
  Thus $(c_n,c_{n-1},\ldots ,c_1)\in \cAp $.

  (4) Equations $\eta (c_1)=-\id $, $\eta (c_1)\eta (c_2)=-\id $ have no
  solutions with $c_1,c_2\in \ndN $ by
  Eqs.~\eqref{eq:etai}, \eqref{eq:eta2}.
  Let now $n=3$ and $c_1,c_2,c_3\in \ndN $. If $c_1,c_2,c_3\ge 2$, then
  $(c_1,c_2,c_3)\notin \cA $ by Cor.~\ref{co:Pringsheim}. Otherwise
  $c_1=c_2=c_3=1$ by Lemma~\ref{le:cApcontr}(3) and Part~(2) of the
  proposition. Relation $(1,1,1)\in \cAp $ holds by
  Eq.~\eqref{eq:eta2} with $i=j=1$. This proves (4).
\end{proof}

By Prop.~\ref{pr:Aprop}(2),(3)
the dihedral group $\DG _n$ of $2n$ elements,
where $n\in \ndN $, acts on sequences
of length $n$ in $\cAp $ by cyclic permutation of the entries and by
reflections.
This action gives rise to an equivalence relation $\sim $
on $\cAp $ by taking the
orbits of the action as equivalence classes.
For brevity we will usually not distinguish between elements of $\cAp $ and
$\cAp /{\sim }$.
By Prop.~\ref{pr:Aprop}(4) there is precisely one element of $\cAp /{\sim }$
of length $3$.

Lemma~\ref{le:cApcontr} suggests to introduce a further equivalence relation
$\approx $ on $\cAp $.
Let $n,m\in \ndN $ with $m\ge n$, and let $c=(c_1,\ldots ,c_n)$,
$d=(d_1,\ldots ,d_m)\in \cAp $.
We write $c\approx 'd$ if and only if
\begin{itemize}
  \item $m=n$, $c\sim d$ or
  \item $m=n+1$, $d=(c_1+1,1,c_2+1,c_3,c_4,\ldots  ,c_n)$.
\end{itemize}

\begin{defin}
  Let $c,d\in \cAp $. Write $c\approx d$ if and only if there exists $k\in
  \ndN $ and a sequence $c=e_1,e_2,\ldots ,e_k=d$ of elements of $\cAp $, such
  that $e_i\approx 'e_{i+1}$ or $e_{i+1}\approx 'e_i$ for all $i\in
  \{1,2,\ldots ,k-1\}$.
\end{defin}

Clearly, $\approx $ is an equivalence relation on $\cAp $.
We are
interested in the equivalence classes of $\cAp /{\approx }$.

\begin{theor}\label{th:rep111}
  The only element of $\cAp /{\approx }$ is $(1,1,1)$.
\end{theor}

\begin{proof}
  Let $n\ge 1$ and $c=(c_1,\ldots ,c_n)\in \cAp $. By Prop.~\ref{pr:Aprop}(4)
  it suffices to prove that if $n\ge 4$, then $c\approx c'$ for some
  $c'=(c'_1,c'_2,\dots ,c'_{n-1})\in \cAp $.
  
  Assume that $n\ge 4$.
  By Cor.~\ref{co:Pringsheim} there exists $i\in
  \{1,\ldots ,n\}$ such that $c_i=1$. By Prop.~\ref{pr:Aprop}(2) and the
  definition of $\approx $ we may assume that $c_2=1$. Now apply
  Lemma~\ref{le:cApcontr}(2),(3) to obtain the desired $c'\in \cAp $.
\end{proof}

\begin{corol}\label{co:CartanSum}
  If $n\in \ndN $, $(c_1,\ldots ,c_n)\in \cAp $, then $\sum _{i=1}^n
  c_i=3(n-2)$.
\end{corol}

\begin{proof}
  The expression $\sum _{i=1}^n c_i -3(n-2)$ is zero for $c=(1,1,1)$ and is an
  invariant of $\approx $.
\end{proof}

\section{Connected root systems of rank two}
\label{sec:crs2}

Throughout this section let $I$ be a set with $|I|=2$,
$A$ a finite set, and
$\cC =\cC (I,A,(\rfl _i)_{i\in I},(C^a)_{a\in A})$
a connected Cartan scheme.
Since $\rfl _i^2=\id $ for all $i\in I$, and $\cC $ is connected,
the object change diagram of $\cC $ is either a chain
\begin{figure}
  \setlength{\unitlength}{.8mm}
  \rule[-3\unitlength]{0pt}{8\unitlength}
  \begin{picture}(34,5)(0,3)
    \put(1,2){\circle*{2}}
    \put(2,2){\line(1,0){10}}
    \put(13,2){\circle*{2}}
    \put(14,2){\line(1,0){5}}
    \put(23,1){\makebox[0pt]{$\cdots $}}
    \put(27,2){\line(1,0){5}}
    \put(33,2){\circle*{2}}
    \put(7,3){\makebox[0pt]{\scriptsize $?$}}
  \end{picture}
  \caption{Chain diagram}
  \label{fi:chain}
  \setlength{\unitlength}{.8mm}
  \begin{picture}(36,18)(0,3)
    \put(1,2){\circle*{2}}
    \put(2,2){\line(1,0){10}}
    \put(13,2){\circle*{2}}
    \put(14,2){\line(1,0){5}}
    \put(23,1){\makebox[0pt]{$\cdots $}}
    \put(27,2){\line(1,0){5}}
    \put(33,2){\circle*{2}}
    \put(1,3){\line(0,1){10}}
    \put(33,3){\line(0,1){10}}
    \put(1,14){\circle*{2}}
    \put(2,14){\line(1,0){10}}
    \put(13,14){\circle*{2}}
    \put(14,14){\line(1,0){5}}
    \put(23,13){\makebox[0pt]{$\cdots $}}
    \put(27,14){\line(1,0){5}}
    \put(33,14){\circle*{2}}
    \put(7,3){\makebox[0pt]{\scriptsize $1$}}
    \put(7,15){\makebox[0pt]{\scriptsize $1$}}
    \put(3,7){\makebox[0pt]{\scriptsize $2$}}
    \put(35,7){\makebox[0pt]{\scriptsize $?$}}
  \end{picture}
  \caption{Cycle diagram}
  \label{fi:cycle}
\end{figure}
(if $\rfl _i$ has a fixed point for some $i\in I$), see Fig.~\ref{fi:chain},
or a cycle, see Fig.~\ref{fi:cycle}.

Recall that an element $w\in \Hom (\Wg (\cC ))$ is called \textit{even} if
$\det (w)=1$.

\begin{lemma}
  The object change
  diagram of $\cC $
  is a cycle if and only if $\End (a)$ contains only even elements
  (for all $a\in A$). 
  \label{le:ocd}
\end{lemma}

\begin{proof}
  If the object change diagram of $\cC $ is a cycle,
  then for all $a\in A$,
  $\End (a)$ consists of the elements $(\s _i\s _j)^{k|A|/2}1_a$,
  where $k\in \ndZ $ and $I=\{i,j\}$. These are all even.
  Otherwise the object change diagram of $\cC $ is a chain, and there exists
  $a\in A$ and $i\in I$ such that $\rfl _i(a)=a$. Then $\End (a)$ is generated
  by $\s _i^a$ and $(\s _j\s _i)^{|A|-1}\s _j^a$ which are odd.
\end{proof}

Assume that $\cC $ admits a finite root system of type $\cC $.
The next proposition explains the relationship between the $m^a_{i,j}$
and the number $|A|$ of objects. For this, we need the following
standard lemma.

\begin{lemma}\label{le:12346}
Let $M\in \GLZ $. If the order $e$ of $M$ is finite, then
\[ -2\le \tr(M)\le 2,\qquad e\in\{1,2,3,4,6\}. \]
\end{lemma}

\begin{propo}\label{pr:mij}
Assume that $\cC $ admits a finite root system of type $\cC $.
Then the numbers $m^a_{i,j}=|R^a_+|$ are equal
for all objects $a$ and $i,j\in I$ with
$i\not=j$.
If the object change diagram is a cycle resp. a chain,
then $m^a_{i,j}=m\frac{|A|}{2}$ resp. $m^a_{i,j}=m|A|$
for some $m\in\{1,2,3,4,6\}$.
\end{propo}

\begin{proof}
We have $m^a_{i,j}=|R^a_+|$ by Def. \ref{de:RSC}. Axiom
(R3) from the same definition implies that all $m^a_{i,j}$ are equal.
Let $d=|A|$ if the object change diagram is a chain and $d=|A|/2$ if it is a
cycle. Then $(\s _j\s _i)^k 1_a\in \End (a)$, $k\in \ndN _0$,
if and only if $k\in \ndN _0d$. Thm.~\ref{th:Coxgr} and
Lemma~\ref{le:12346} give that $md=m^a_{i,j}$ for some $m\in \{1,2,3,4,6\}$.
\end{proof}

We are going to give a characterization of
finite connected irreducible root systems of type $\cC $. First we analyze
root systems with simply connected Cartan schemes.

\begin{lemma}\label{le:ocdscrs}
  Assume that $\cC $ is simply connected
  and that $\rsC $ is a finite root system of type $\cC $.
  Then the object change diagram of $\cC $ is a cycle with $|R^a|$
  vertices, where $a\in A$.
\end{lemma}

\begin{proof}
  Since $\cC $ is simply connected, $\End (a)=\{1\}$ for all $a\in A$.
  By Lemma~\ref{le:ocd} the object change diagram of $\cC $ is a cycle.
  Now
  \[ |\Hom (\Wg (\cC ))1_a|=|A|\cdot |\End (a)| \]
  since $\cC $ is connected. Again,
  $\cC $ is simply connected, and hence $|A|=|\Hom (\Wg (\cC ))1_a|$.
  This is equal to $2|R^a_+|$ by Thm.~\ref{th:Coxgr}, since $|I|=2$.
\end{proof}

\begin{propo}\label{pr:rstoc}
  Assume that $I=\{i,j\}$
  and that $\rsC $ is a finite irreducible root system of type $\cC $.
  Let $a\in A$ and $n=|R^a_+|$.
  Let $a_1,a_2,\dots $, $a_{2n}\in A$ and $c_1,c_2,\dots,c_{2n}\in \ndZ $
  such that
  \begin{align}\label{eq:acdef}
    \begin{aligned}
      a_{2r-1}=&(\rfl _j\rfl _i)^{r-1}(a),&
      a_{2r}=&\rfl _i(\rfl _j\rfl _i)^{r-1}(a),\\
      c_{2r-1}=&-c^{a_{2r-1}}_{i j}, &
      c_{2r}=&-c^{a_{2r}}_{j i}
    \end{aligned}
  \end{align}
  for all $r\in \{1,2,\ldots ,n\}$.
  Then $(c_1,c_2,\ldots ,c_n)\in \cAp $, $c_{n+r}=c_r$
  for all $r\in \{1,2,\ldots ,n\}$, and $\rfl _j(a_{2n})=a$.
\end{propo}

\begin{proof}
  For all $r\in \ndZ $ let $i_r\in I$ such that
  $i_r=i$ for $r$ odd and $i_r=j$ for $r$ even.
  Let $\theta _{2r-1}=\s _i^{a_{2r-1}}\tau $,
  $\theta _{2r}=\tau \s _j^{a_{2r}} \in \SLZ $ for all
  $r\in \{1,\ldots ,n\}$. Then $\theta _r=\eta (c_r)$ for all
  $r\in \{1,\ldots ,2n\}$. Since $\rsC $ is irreducible,
  $c_r>0$ for all $r$.
  By \cite[Lemmas\,4,7]{a-HeckYam08},
  $\ell (\s ^{a_n}_{i_n}\cdots \s ^{a_2}_{i_2}\s ^a_{i_1})=n$. Hence
  \[
    \s ^{a_n}_{i_n}\cdots \s ^{a_2}_{i_2}\s ^a_{i_1}(\{\al _1,\al _2\})
    =\{-\al _1,-\al _2\}
  \]
  by \cite[Lemma\,8(iii)]{a-HeckYam08}.
  Thus $\theta _n\cdots \theta _2\theta _1(\{\al _1,\al _2\})
  =\{-\al _1,-\al _2\}$, and since $\det \theta _r=1$ for all $r$, we conclude
  that $\theta _n\cdots \theta _2\theta _1=-\id $. Hence
  $(c_n,\ldots ,c_2,c_1)\in \cA $.

  Clearly, if $2\le r\le n$, then the first column of
  $\theta _n\cdots \theta _{r+1}\theta _r$ has
  nonnegative entries if and only if $\s _{i_n}\cdots \s _{i_{r+1}}\s
  _{i_r}^{a_r}(\al _{i_{r-1}})$ is a positive root. The latter is true by
  \cite[Lemma\,4]{a-HeckYam08}, and hence $(c_n,\ldots ,c_2,c_1)\in \cAp $.
  Then $(c_1,c_2,\ldots ,c_n)\in \cAp $ by Prop.~\ref{pr:Aprop}(3).

  Replacing in the construction $a$ by $a_2$ and $i$ by $j$,
  we obtain that $(c_2,\ldots ,c_n,c_{n+1})\in \cAp $.
  Then $\eta (c_1)^{-1}=-\eta (c_2)\cdots \eta (c_n)=\eta (c_{n+1})^{-1}$, and
  hence $c_1=c_{n+1}$. Thus $c_{n+r}=c_r$ for all $r\in \{1,2,\ldots ,n\}$
  by induction on $r$. Finally,
  $\rfl _j(a_{2n})=(\rfl _j\rfl _i)^n(a)=a$ by (R4).
\end{proof}

The construction in Prop.~\ref{pr:rstoc} associates to any pair $(i,a)\in
I\times A$ a sequence $(c_1,c_2,\ldots ,c_n)\in \cAp $.
This defines a map
\[ \Phi :I\times A\to \cAp .\]
Prop.~\ref{pr:rstoc} gives immediately, that
\begin{align}
  \Phi (j,a)=(c_n,c_{n-1},\dots ,c_1), \quad
  \Phi (j,\rfl _i(a))=(c_2,c_3,\dots ,c_n,c_1).
  \label{eq:Phiim}
\end{align}
Thus, by definition of $\sim $,
the induced map $\Phi :I\times A\to \cAp /{\sim }$ is
constant. But we can say more.

\begin{theor}\label{th:ctors}
  Let $n\in \ndN $ and $c=(c_1,c_2,\ldots ,c_n)\in \cAp $. Then there is a
  unique (up to equivalence)
  finite connected simply connected irreducible root system $\rsC $ of rank
  two such that $c\in \mathrm{Im}\,\Phi $. 
\end{theor}

\begin{proof}
  Assume that $c\in \cAp $, $\rsC $ is a
  connected irreducible root system of rank two, $i\in I$, and $a\in A$ such
  that $\Phi (i,a)=c$. If $\rsC $ is simply connected, then
  by Lemma~\ref{le:ocdscrs} and Prop.~\ref{pr:rstoc}
  the object change diagram of $\rsC $ is a cycle
  and $|A|=2n$. The Cartan matrices $C^a$ and the sets $R^a$, where $a\in A$,
  are then
  uniquely determined by the construction in Prop.~\ref{pr:rstoc}. Thus $\rsC
  $ is uniquely determined. We describe $\rsC $ explicitly.
  
  Let $I=\{i,j\}$ and $A=\{a_1,\ldots ,a_{2n}\}$ a set with $2n$ elements.
  Define $\rfl _i,\rfl _j:A\to A$ such that
  \begin{align} \label{eq:rhodef}
    \begin{aligned}
      \rfl _i(a_{2r-1})=&a_{2r}, & \rfl _i(a_{2r})=&a_{2r-1},\\
      \rfl _j(a_{2r})=&a_{2r+1}, & \rfl _j(a_{2r+1})=&a_{2r}
    \end{aligned}
  \end{align}
  for all $r\in \{1,2,\ldots ,n\}$, where $a_{2n+1}=a_1$.
  Then $\rfl _i^2=\rfl _j^2=\id $. Let $c_{ln+r}=c_r$ for all $r\in
  \{1,2,\ldots ,n\}$ and $l\in \ndZ $, and define
  \begin{align}
    C^{a_{2r-1}}=
    \begin{pmatrix}
      2 & -c_{2r-1} \\ -c_{2r-2} & 2
    \end{pmatrix}, \quad
    C^{a_{2r}}=
    \begin{pmatrix}
      2 & -c_{2r-1} \\ -c_{2r} & 2
    \end{pmatrix}
    \label{eq:cCseq}
  \end{align}
  for all $r\in \{1,2,\ldots ,n\}$.
  Since $c_r\in \ndN $ for all $r\in \{1,2,\ldots ,2n\}$,
  the matrices $C^{a_r}$ satisfy (M1) and (M2). Since also (C1)
  and (C2) hold, $\cC =\cC (I,A,(\rfl _i,\rfl _j),(C^a)_{a\in A})$
  is a connected Cartan scheme.

  Now define
  \begin{align*}
    R^{a_{2r-1}}=&\left\{ \pm \eta (c_{2r-1})\eta (c_{2r})\cdots
    \eta (c_{2r-2+l})
    \begin{pmatrix} 1 \\ 0 \end{pmatrix} \,\big|\,0\le l\le n-1 \right \},\\
    R^{a_{2r}}=&\left\{ \pm \tau \eta (c_{2r})\eta (c_{2r+1})\cdots
    \eta (c_{2r+l-1})
    \begin{pmatrix} 1 \\ 0 \end{pmatrix} \,\big|\,0\le l\le n-1 \right \}
  \end{align*}
  for all $r\in \{1,2,\dots ,n\}$.
  Note that $|R^a_+|=n$ for all $a\in A$.
  Indeed, otherwise $\eta (c_r)\eta (c_{r+1})\cdots \eta (c_{r+l-1})
  {1 \choose 0} = {1 \choose 0}$ for some $r\in \{1,2,\ldots ,2n\}$ and
  $l\in \{1,2,\ldots ,n-1\}$. Then 
  \[ \eta (c_{r+1})\eta (c_{r+2})\cdots \eta (c_{r+l-1})
  {1 \choose 0} = \eta (c_r)^{-1}{1 \choose 0} = {0 \choose -1}, \]
  a contradiction to
  Prop.~\ref{pr:Aprop}(1),(2).

  Axiom~(R1) is fulfilled by Prop.~\ref{pr:Aprop}(2).
  Let $r\in \{1,2,\ldots ,2n\}$.
  Equation $\eta (c_r)\eta (c_{r+1})\cdots \eta (c_{r+n-1})=-\id $
  implies, that
  \[ \eta (c_r)\eta (c_{r+1})\cdots \eta (c_{r+n-2})=-\eta
  (c_{r+n-1})^{-1}, \]
  and hence $\pm \al _1,\pm \al _2\in R^{a_r}$.
  Since $\tau ,\eta (l)\in \SLZ $ for all $l\in \ndZ $, we get (R2).
  (R4) holds by Eq.~\eqref{eq:rhodef},
  since $|R^a_+|=n$ for all $a\in A$.
  
  Now we prove (R3). Let $r\in \{1,2,\ldots ,2n\}$.
  Then $\s _i^{a_r}=\eta (-c^{a_r}_{ij})\tau =\tau \eta (-c^{a_r}_{ij})^{-1}$
  by Eqs.~\eqref{eq:cCseq}, \eqref{eq:etatau}.
  If $r$ is odd, then
  \begin{align*}
    &\s _i^{a_r}(R^{a_r})\\
    &=
    \tau \eta (c_r)^{-1}\left(\left\{ \pm \eta (c_r)\eta (c_{r+1})\cdots
    \eta (c_{r+l-1}){1\choose 0}  \,\big|\,
      0\le l\le n-1 \right \}\right)\\
      & \quad \subset R^{a_{r+1}}=R^{\rfl _i(a_r)},
  \end{align*}
  and if $r$ is even, then
  \begin{align*}
    &\s _i^{a_r}(R^{a_r})\\
    &=
    \eta (c_{r-1})\tau \left(\left\{ \pm \tau \eta (c_r)\eta (c_{r+1})\cdots
    \eta (c_{r+l-1}){1\choose 0}  \,\big|\,
      0\le l\le n-1 \right \}\right)\\
      & \quad \subset R^{a_{r-1}}=R^{\rfl _i(a_r)}.
  \end{align*}
  Similarly, $\s _j^{a_r}=\tau \eta (c_{r-1})$ for odd $r$ and
  $\s _j^{a_r}=\eta (c_r)^{-1}\tau $ for even $r$. Hence
  $\s _j^{a_r}(R^{a_r})\subset R^{\rfl _j(a_r)}$,
  (R3) holds, and $\rsC $
  is a finite irreducible root system of type $\cC $. It is simply connected,
  since $(c_1,\dots ,c_n)\in \cA $ and 
  $|\Hom (\Wg (\cC ))1_{a_1}|=2n=|A|$.
  Finally, $\Phi (i,a_1)=(c_1,\ldots ,c_n)$ because of
  Eqs.~\eqref{eq:acdef}, \eqref{eq:rhodef}, and \eqref{eq:cCseq}.
\end{proof}

\begin{corol}\label{co:minusone}
  Assume that there is a finite root system $\rsC $ of type $\cC $.
  Then there are $a\in A$ and $i,j\in I$ with $i\not=j$ such that
  $c^a_{ij}=0$ or $c^a_{ij}=-1$.
\end{corol}

\begin{proof}
  If $\rsC $ is not irreducible, then $C^a_{ij}=0$ for all $a\in A$ and
  $i,j\in I$ with $i\not=j$, see the end of Sect.~\ref{sec:CS}.
  Otherwise Prop.~\ref{pr:rstoc} gives that the negatives of the entries of
  the Cartan matrices of $\cC $ give rise to a
  sequence $(c_1,\dots ,c_n)\in \cA ^+$.
  By Cor.~\ref{co:Pringsheim}, this sequence has an entry $1$, and the
  corollary is proven.
\end{proof}

\begin{remar}
  The assumption in Cor.~\ref{co:minusone} can be weakened for example by
  requiring only that $\Wg (\cC )$ is finite. We don't work out the details,
  since we are mainly interested in Cartan schemes admitting (finite)
  root systems.
\end{remar}

We are going to give a very effective algorithm
to decide if our given
connected Cartan scheme $\cC $ admits a finite irreducible
root system. The central notions towards this will be the
characteristic sequences and centrally symmetric Cartan schemes.
Our algorithm can also be used to get a more precise classification of
root systems of rank two, for example in form of explicit lists for a
given number of objects.

\begin{defin}\label{de:centsym}
  Assume that the object change diagram of $\cC $ is a cycle.
  Let $i\in I$, $a\in A$, and define
  $a_1,\ldots ,a_{|A|}\in A$ and $c_1, \dots , c_{|A|}\in \ndN _0$ by
  \begin{align*}
    a_{2k-1}=& (\rfl _j\rfl _i)^{k-1}(a), &
    a_{2k}=& (\rfl _i\rfl _j)^{k-1}\rfl _i(a),\\
    c_{2k-1}=&-c^{a_{2k-1}}_{ij}, &
    c_{2k}=&-c^{a_{2k}}_{ji}
  \end{align*}
  for all $k\in \{1,2,\dots ,|A|/2\}$, where $I=\{i,j\}$.
  Then $(c_1,c_2,\dots ,c_{|A|})$ is called the \textit{characteristic
  sequence of} $\cC $ with respect to $i$ and $a$.
  The Cartan scheme $\cC $ is termed \textit{centrally symmetric}, if
  $c_k=c_{k+|A|/2}$ for all $k\in \{1,2,\dots ,|A|/2\}$.
  In this case we write also $(c_1,c_2,\dots ,c_{|A|/2})^2$ for
  $(c_1,c_2,\dots ,c_{|A|})$.
\end{defin}

\begin{remar}
  Let $(c_1,c_2,\dots ,c_{|A|})$ be the characteristic
  sequence of $\cC $ with respect to $i$ and $a$.
  Then the characteristic sequences with respect to $j$ and $a$ and $i$ and
  $\rfl _i(a)$, respectively, are
  $(c_{|A|},c_{|A|-1},\dots ,c_1)$ and
  $(c_1,c_{|A|},c_{|A|-1},\dots ,c_3,c_2)$, respectively.
  Thus if $\cC $ is centrally
  symmetric with respect to $i$ and $a$, it is also centrally symmetric
  with respect to $j$ and $a$ and $i$ and $\rfl _i(a)$, respectively.
  Since $\cC $ is connected,
  this means that $\cC $ being centrally symmetric is independent of
  the choice of $i\in I$ and $a\in A$.
\end{remar}

\begin{remar}\label{re:charseqcA}
  Characteristic sequences must not be confused with elements of $\cA $ or
  $\cAp $. Their precise relationship will not be needed in the sequel, so we
  don't work it out in detail.
\end{remar}

\begin{remar}\label{re:charseqtocs}
  Let $n\in \ndN $ and let $c=(c_1,c_2,\dots ,c_{2n})$ be a sequence of
  positive integers.  By axioms (M1) and (C2) there is
  a unique (up to equivalence) connected Cartan scheme $\cC $ with
  object change diagram a cycle, such that the characteristic sequence of $\cC
  $ (with respect to some $i\in I$ and $a\in A$) is $c$.
\end{remar}

\begin{remar}\label{re:chseqsc}
  Assume that $\cC $ is simply connected,
  and that there exists a finite irreducible root system of type $\cC $.
  Then $\cC $ is centrally symmetric by Lemma~\ref{le:ocdscrs}
  and Prop.~\ref{pr:rstoc}.
\end{remar}

\begin{remar}\label{re:chseqcov}
  Assume that the object change diagram of $\cC $ is a cycle.
  By Lemma~\ref{le:ocd} and Prop.~\ref{pr:cover} the object change diagram of
  an $n$-fold covering $\cC '$ of $\cC $, where $n\in \ndN $, is a cycle.
  The characteristic sequence of $\cC '$
  is just the $n$-fold repetition of the characteristic sequence of $\cC $.
  Thus an $n$-fold covering of $\cC $ is centrally symmetric if and
  only if $\cC $ is centrally symmetric or $n$ is even.
\end{remar}

\begin{lemma}
  Assume that there exists a
  finite irreducible root system of type $\cC $.
  Suppose that the object change diagram of $\cC $ is a chain.
  Then there is a unique double covering
  $\cC '$
  of $\cC $ and a finite irreducible root system of type $\cC '$
  such that the object change
  diagram of $\cC '$ is a cycle.
  \label{le:chain}
\end{lemma}

\begin{proof}
  By assumption there exists $a\in A$ and $i\in I$ such that $\rfl _i(a)=a$.
  Then $\End (a)$ is generated by $\s _i^a$ and
  $\tau ^a=(\s _j\s _i)^{|A|-1}\s _j^a$, where $I=\{i,j\}$.
  Since $\s _i^a$, $\tau ^a$
  are reflections, for the subgroup $U=\langle \s _i^a\tau ^a\rangle \subset
  \End (a)$ we obtain that $[\End (a):U]=2$, and $U$ consists of even elements.
  By Prop.~\ref{pr:cover}(2) there exists a unique
  double covering $\cC '$ of $\cC $ satisfying Axiom~(C3) such that $\End
  (a')\simeq U$ for all $a'\in A'$. By Lemma~\ref{le:ocd}
  the object change diagram of $\cC '$ is a cycle. The uniqueness of $\cC '$
  holds, since $U$ is the unique subgroup of $\End (a)$ consisting of even
  elements and satisfying $[\End (a):U]=2$.
  The existence of a finite irreducible
  root system of type $\cC '$ follows from Prop.~\ref{pr:covrs}(2).
\end{proof}

\begin{remar}\label{re:chaincover}
  If $\cC '$ is a Cartan scheme with object change
  diagram a cycle, then $\cC '$ is the double covering of a Cartan scheme with
  object change diagram a chain if and only if there exist $i\in I'$,
  $a\in A'$, such that the characteristic sequence of $\cC '$ with respect to
  $i$ and $a$ is of the form
  $(c_1,\dots ,c_n,c_{n+1},c_n,c_{n-1},\dots ,c_2)$ with $n=|A'|/2$ and
  $c_1,\dots,c_{n+1}\in \ndN _0$.
\end{remar}

\begin{lemma}
  Assume that there exists a
  finite irreducible root system of type $\cC $.
  Suppose that the object change diagram of $\cC $ is a cycle, and that $\cC $
  is not centrally symmetric.
  Then there is a unique double covering $\cC '$ of $\cC $ which admits a
  (finite irreducible) root system.
  The Cartan scheme $\cC '$ is centrally symmetric.
  \label{le:cycle}
\end{lemma}

\begin{proof}
  Since the object change diagram of $\cC $ is a cycle, $\End (a)$ is cyclic
  for all $a\in A$.
  The universal covering of $\cC $ is centrally symmetric by
  Rem.~\ref{re:chseqsc}.
  Since $\cC $ is not centrally symmetric, $|\End (a)|$ is even
  by Rem.~\ref{re:chseqcov} and Prop.~\ref{pr:cover}(2).
  By Prop.~\ref{pr:cover}(2) there is a unique double
  covering $\cC '$ of $\cC $ satisfying (C3).
  It admits a finite irreducible root system of type $\cC '$
  by Prop.~\ref{pr:covrs}(2).
  All coverings of $\cC $ admitting a root system fulfill (C3).
  Hence $\cC '$ is the only double covering of $\cC $ admitting a root
  system. This $\cC '$ is centrally symmetric by Rem.~\ref{re:chseqcov}.
\end{proof}

\begin{remar}\label{re:cyclecover}
  Let $\cC '$ be a Cartan scheme with object change
  diagram a centrally symmetric cycle, and $n=|A'|$. Then
  $\cC '$ is the double covering of a Cartan scheme with
  object change diagram a not centrally symmetric cycle if and only if
  $n\in 4\ndN $, and
  with respect to one (equivalently, all) pair $(i',a')\in I'\times A'$
  the characteristic sequence of $\cC '$ is not of the form
  \[ (c_1,c_2,\dots ,c_{n/4},c_1,c_2,\dots ,c_{n/4})^2, \]
  where $c_1,\dots,c_{n/4}\in \ndN _0$.
\end{remar}

In order to decide if a given connected Cartan scheme admits a finite
root system, Lemmas~\ref{le:chain} and \ref{le:cycle} allow to concentrate
on centrally symmetric Cartan schemes. Further, since the classification of
finite root systems with at most three objects is known, see \cite{p-CH08},
we may assume that
the Cartan scheme has at least $4$ objects.

For any matrix $C$, let $C^\trans $ denote the transpose of $C$.

\begin{theor}\label{th:cycleclass}
  Let $\cC =\cC (I,A,(\rfl _i)_{i\in I},(C^a)_{a\in A})$
  be a connected centrally symmetric Cartan scheme with $|A|\ge 4$.
  
  (1) Assume that the characteristic sequence of $\cC $ contains $0$. Then
  $c^a_{ij}=0$ for all $a\in A$ and $i,j\in I$ with $i\not=j$.
  Moreover, $\cC $ admits a finite
  root system if and only if $|A|=4$.

  (2) If all entries of the characteristic sequence of $\cC $ are at least
  two, then $\cC $ does not admit a finite root system.

  (3) Assume that the characteristic sequence of $\cC $ is of the form
  \[ c=(c_1,1,c_3,c_4,\dots, c_{|A|/2})^2. \]
  If $c_1=1$ or $c_3=1$, then there is a finite root system of type $\cC $
  if and only if $|A|=6$ and $c_1=c_3=1$.
  If $c_1>1$ and $|A|=4$, then
  there is a finite root system of type $\cC $ if and only if
  $c_1\in \{2,3\}$.
  If $c_1>1$, $c_3>1$, and $|A|\ge 6$,
  then there is a finite root system of type $\cC $
  if and only if the Cartan scheme with object change diagram a cycle with
  $|A|-2$ edges and with characteristic sequence
  \begin{align}
    (c_1-1,c_3-1,c_4,\dots, c_{|A|/2})^2
    \label{eq:redseq}
  \end{align}
  admits a finite root system.
\end{theor}

\begin{proof}
  (1) follows from (M2), (C2), and (R4), and (2) from Cor.~\ref{co:minusone}.

  (3) If $c_1=1$ or $c_3=1$, then there exists $a\in A$ such that
  $c^a_{ij}=c^a_{ji}=-1$, where $I=\{i,j\}$. Then \cite[Lemma\,4.8]{p-CH08}
  gives that $m^a_{i,j}=3$ and $c_r=1$ for all $r\in \{1,3,4,\dots,|A|/2\}$.
  By (R4) we get $|A|=6$.

  Assume next that $c_1>1$ and $|A|=4$.
  Then $C^a=C^b$ for all $a,b\in A$, and hence $\cC $ admits a finite root
  system if and only if $C^a$ is of finite type and (R4) holds (cf.
  \cite[Thm.\,3.3]{p-CH08}), that is, $c_1\in \{2,3\}$.

  Finally, assume that $c_1>1$, $c_3>1$, $|A|\ge 6$, and
  $\cC $ admits a finite root system.
  By Prop.~\ref{pr:covrs},
  the universal covering $\cC '$ of $\cC $ admits a finite root system.
  Hence $A'$ is finite by (C1) and (R4). Therefore
  $\End (a)\subset \Hom (\Wg (\cC ))$ is finite for all $a\in A$ by
  Eq.~\eqref{eq:piprop2}.  Let $m=|\End (a)|$.
  Rem.~\ref{re:chseqcov}
  and Lemma~\ref{le:ocdscrs} tell that
  the object change diagram of $\cC '$ is a centrally symmetric cycle, and
  the characteristic sequence of $\cC '$ is an $m$-fold repetition of $c$.
  Let
  \[ \tilde{c}=(c_1,1,c_3,c_4,\dots, c_{|A|/2}). \]
  By Prop.~\ref{pr:rstoc} 
  the $m$-fold repetition of $\tilde{c}$ is an element of $\cAp $.
  Since $|A|\ge 6$, Lemma~\ref{le:cApcontr}(2) gives that
  the $m$-fold repetition of
  \[ \tilde{c}'=(c_1-1,c_3-1,c_4,\dots, c_{|A|/2}) \]
  is in $\cAp $.
  Let $\cC ''$ be the connected simply connected
  Cartan scheme which corresponds to the
  $m$-fold repetition of $\tilde{c}'$ via Thm.~\ref{th:ctors}.
  It admits a finite root system.
  Now $\cC ''$ is the $m$-fold covering of a Cartan scheme $\cC '''$
  with characteristic sequence given in Eq.~\eqref{eq:redseq}.
  Hence Prop.~\ref{pr:covrs} gives that $\cC '''$ admits a finite root
  system.
  
  We have shown that if $\cC $ admits a finite root system,
  then also $\cC '''$. The proof of the converse goes in the same way, and we
  are done.
\end{proof}

\begin{examp}
Consider the connected Cartan scheme $\cC $ of rank two with $4$ objects,
object change diagram a cycle and
characteristic sequence $(5,1,2,2)$.
To check that $\cC $ admits a finite root system, consider the double
covering $\cC '$
corresponding to the characteristic sequence $(5,1,2,2)^2$.
By Prop.~\ref{pr:covrs}, $\cC $ admits a finite root system if and only if
$\cC '$ does.
Thm.~\ref{th:cycleclass}(3) allows to replace $\cC '$ by the Cartan scheme
with characteristic sequence
$(4,1,2)^2$ respectively $(3,1)^2$. Thus $\cC $ admits a finite root
system.

If we start with the characteristic sequence
$(5,1,2,3)$ for $\cC $, then the analogous
arguments produce the characteristic sequences $(5,1,2,3)^2$,
$(4,1,3)^2$ and $(3,2)^2$, and then $\cC $ does not admit a finite root
system by Thm.~\ref{th:cycleclass}(2).
\end{examp}

\section{Bounds}
\label{sec:bounds}

Let
$\cC =\cC (I,A,(\rfl _i)_{i\in I},(C^a)_{a\in A})$ be a
connected Cartan scheme of rank two
admitting a finite irreducible root system of type $\cC $.
Then $A$ is finite by (C1) and (R4).
Let $-q=-q(\cC )$ denote the sum of all non-diagonal
entries of the Cartan matrices of $\cC$, and
$h=|\End (a)|$ for an $a\in A$. Then $|\End (b)|=h$ for all $b\in A$, since
$\cC $ is connected.

\begin{theor}\label{th:h}
  We have $h(6|A|-q)=24$ and
  \[ |R^a_+|=\frac{h|A|}{2}=\frac{12|A|}{6|A|-q}. \]
\end{theor}

\begin{proof}
  The universal covering $\cC '$ of $\cC $ has $h|A|$ objects
  by Eq.~\eqref{eq:piprop2}, and $q(\cC ')/4=3(h|A|/2-2)$
  by Prop.~\ref{pr:rstoc} and Cor.~\ref{co:CartanSum}.
  Since $q(\cC ')=hq(\cC )$, we obtain that
  $hq=6(h|A|-4)$. Hence $h(6|A|-q)=24$.
  Lemma~\ref{le:ocdscrs} tells that $|R^a_+|=h|A|/2$. This yields the claim.
\end{proof}

\begin{remar}
  Prop.~\ref{pr:mij} and Thm.~\ref{th:h} give that
  $h\in \{1,2,3,4,6\}$ if the object change diagram of $\cC $ is a cycle, and
  $h/2\in \{1,2,3,4,6\}$ if it is a chain. But this result could have been
  obtained much easier. Nevertheless, Thm.~\ref{th:h} gives a restriction for
  $q=6|A|-24/h$ for given number $|A|$ of objects in a finite irreducible
  root system.
\end{remar}

Next we give sharp bounds for the entries of the Cartan matrices.
%

\begin{propo}\label{pr:bound}
  Assume that $|A|\ge 2$.
  Let $c\le 0$ be an entry of $C^a$ for some $a\in A$.
  If the object change diagram is a cycle resp. a chain,
  then $|c|\le |A|+1$ resp. $|c|\le 2|A|+1$.
\end{propo}

\begin{proof}
Assume first that
the object change diagram of $\cC $ is a cycle.
If $|A|\ge 4$ and
$\cC $ is centrally symmetric,
then Thm.~\ref{th:cycleclass}(2),(3) yields by induction on $|A|$, that
$|c|\le |A|/2+1$.
If $\cC $ is not centrally symmetric,
then by Lemma~\ref{le:cycle} there exists a double covering of $\cC $
which is centrally symmetric. Hence $|c|\le |A|+1$.

If the object change diagram of $\cC $ is a chain, then
by Lemma~\ref{le:chain} there exists a double covering of $\cC $
which has a cycle as object change diagram.
Hence $|c|\le 2|A|+1$.
\end{proof}


\begin{propo} \label{pr:entries}
  For all $n\ge 1$ there exist finite connected
  irreducible root systems
  $\rsC $ of rank two with $|A|=2n$ and object change diagram a cycle
  resp. $|A|=n$ and object change diagram a chain
  such that $-(2n+1)$ is an entry in a Cartan matrix
  $C^a$, $a\in A$.
\end{propo}


\begin{proof}
  For $n=1$ the claim follows from \cite[Prop.\,5.2]{p-CH08}.

Thm.~\ref{th:cycleclass} tells that for all $n\ge 2$ the Cartan scheme
$\cC _n$ with $4n$ objects, object change diagram a cycle, and
characteristic sequence
\begin{align}
  (3,\underbrace{2,2,\dots ,2}_{n-2 \text{ times}},1,2n+1,1,
  \underbrace{2,2,\dots ,2}_{n-2 \text{ times}})^2
  \label{eq:symmcharseq}
\end{align}
admits a finite irreducible root system with $|A|=4n$.
Indeed, if $n=2$, then using Thm.~\ref{th:cycleclass}(3)
we can transform the sequence $(3,1,5,1)^2$ first to $(2,4,1)^2$.
By changing the reference object, the latter is equivalent to
$(4,1,2)^2$, and
using Thm.~\ref{th:cycleclass}(3) we may reduce it to $(3,1)^2$.
If $n>2$, then 
using Thm.~\ref{th:cycleclass}(3)
we may transform the sequence in \eqref{eq:symmcharseq}
in two steps, first to
\[ (3,\underbrace{2,2,\dots ,2}_{n-3 \text{ times}},1,2n,1,
\underbrace{2,2,\dots ,2}_{n-2 \text{ times}})^2, \]
and then to
\[ (3,\underbrace{2,2,\dots ,2}_{n-3 \text{ times}},1,2n-1,1,
\underbrace{2,2,\dots ,2}_{n-3 \text{ times}})^2. \]
By induction on $n$ we obtain that $\cC _n$ admits a finite irreducible
root system.
By Rem.~\ref{re:cyclecover},
$\cC _n$ is the double covering of a Cartan scheme $\cC '_n$ with
$2n$ objects, object change diagram a cycle, and characteristic sequence
\[  (3,\underbrace{2,2,\dots ,2}_{n-2 \text{ times}},1,2n+1,1,
  \underbrace{2,2,\dots ,2}_{n-2 \text{ times}}). \]
By Prop.~\ref{pr:covrs}, $\cC _n'$ admits a finite irreducible root system
$\rsC '$, and $\rsC '$ is such a root system we are looking for.
By Rem.~\ref{re:chaincover}, $\cC '_n$ is the double covering of a Cartan
scheme $\cC ''_n$ with $n$ objects and object change diagram a chain.
By Prop.~\ref{pr:covrs}, $\cC ''_n$ admits a finite irreducible root system
$\rsC ''$, and the proposition is proven.
\end{proof}

\begin{corol}
  Any $c\in\NN$ occurs as the negative of an entry of a Cartan matrix of a
  finite connected irreducible root system of rank two.
\end{corol}

\begin{proof}
  For even $c$ use the appropriate intermediate step in the proof of
  Prop.~\ref{pr:entries}.
\end{proof}

\begin{corol}
  For $r,n\in\NN$, there are only finitely many
  finite root systems $\rsC $ of rank $r$ with $n$ objects.
\end{corol}

\begin{proof}
  Let $I$, $A$ be finite sets with $|I|=r$ and $|A|=n$, and
  let $\rsC $ be a
  finite root system of rank $r$ with object set $A$.
  For all $i,j\in I$ with $i\not=j$ the restriction $\rsC |_{\{i,j\}}$, see
  \cite[Def.\,4.1]{p-CH08}, is a finite root system of rank two.
  Hence the entries of the
  Cartan matrices of $\rsC $ are bounded by $2|A|+1$ by Prop.~\ref{pr:bound}.
  Since for all $i\in I$, $\rfl _i$ is one of finitely many permutations of
  $A$, and since finite root systems are uniquely determined by their Cartan
  scheme, the claim is proven.
\end{proof}


\providecommand{\bysame}{\leavevmode\hbox to3em{\hrulefill}\thinspace}
\providecommand{\MR}{\relax\ifhmode\unskip\space\fi MR }
\providecommand{\MRhref}[2]{%
  \href{http://www.ams.org/mathscinet-getitem?mr=#1}{#2}
}
\providecommand{\href}[2]{#2}

\end{document}